\documentclass[letterpaper, 10 pt, conference]{ieeeconf} 
\IEEEoverridecommandlockouts 
\usepackage{graphics} 
\usepackage{epsfig} 
\usepackage{amsmath} 
\usepackage{amssymb} 
\usepackage{subfigure,color,url}

\newcommand{\bracket}[1]{\ensuremath{\left[ #1 \right]}}
\newcommand{\braces}[1]{\ensuremath{\left\{ #1 \right\}}}

\newcommand{\refeqn}[1]{(\ref{eqn:#1})}

\newcommand{\tr}[1]{\mbox{tr}\ensuremath{\negthickspace\bracket{#1}}}

\newcommand{\SO}{\ensuremath{\mathsf{SO(3)}}}

\newcommand{\so}{\ensuremath{\mathfrak{so}(3)}}

\renewcommand{\Re}{\ensuremath{\mathbb{R}}}
\newcommand{\Sph}{\ensuremath{\mathsf{S}}}

\newcommand{\D}{\ensuremath{\mathbf{D}}}

\newtheorem{prop}{Proposition}

\title{\LARGE \bf Geometric Stabilization of a Quadrotor UAV with a Payload\\ Connected by Flexible Cable}

\author{Farhad A. Goodarzi, Daewon Lee, and Taeyoung Lee$^{*}$
\thanks{Farhad A. Goodarzi, Daewon Lee, and Taeyoung Lee, Mechanical and Aerospace Engineering, The George
Washington University, Washington DC 20052
{\tt\small {\{fgoodarzi,daewonlee,tylee}\}@gwu.edu}}%
\thanks{$^*$This research has been supported in part by NSF under the grants CMMI-1243000 (transferred from 1029551), CMMI-1335008, and CNS-1337722.}}

\begin{document}

\maketitle

\pagestyle{empty}

\begin{abstract}
This paper deals with dynamics and control of a quadrotor UAV with a payload that is connected via a flexible cable, which is modeled as a system of serially-connected links. It is shown that a coordinate-free form of equations of motion can be derived for an arbitrary number of links according to Lagrangian mechanics on a manifold. A geometric nonlinear control system is also presented to asymptotically stabilize the position of the quadrotor while aligning the links to the vertical direction. These results will be particularly useful for aggressive load transportation that involves deformation of the cable. The desirable properties are illustrated by a numerical example and a preliminary experimental result. 
\end{abstract}

\section{Introduction}

Quadrotor unmanned aerial vehicles (UAV) have been envisaged for various applications such as surveillance or mobile sensor networks as well as for educational purposes. 
Areal transportation of a cable-suspended load has been studied traditionally for helicopters~\cite{CicKanJAHS95,BerPICRA09}. 
Recently, small-size single or multiple autonomous vehicles are considered for load transportation and deployment~\cite{PalCruIRAM12,MicFinAR11,MazKonJIRS10,MelShoDARSSTAR13}, and trajectories with minimum swing of payload are generated~\cite{ ZamStaJDSMC08, SchMurIICRA12, PalFieIICRA12}. 



However, these results are based on the common assumption that the cable connecting a quadrotor UAV to a payload is always taut. Therefore, they cannot be applied to aggressive, rapid load transportations where the cable is deformed or the tension along the cable is low, thereby restricting its applicability.
It is challenging to incorporate the effects of a deformable cable, since the dimension of the configuration space becomes infinite. Finite element approximation of a cable often yields complicated equations of motion that make dynamic analysis and controller design extremely difficult.

Recently, a coordinate-free form of the equations of motion for a chain pendulum connected a cart that moves on a horizontal plane is presented according to Lagrangian mechanics on a manifold~\cite{LeeLeoPICDC12}. By following the similar approach, in this paper, the cable is modeled as an arbitrary number of links with different sizes and masses that are serially-connected by spherical joints, as illustrated at Figure \ref{fig:Quad}. The resulting configuration manifold is the product of the special Euclidean group for the position and the attitude of the quadrotor, and a number of two-spheres that describe the direction of each link. We present Euler-Lagrange equations of the presented quadrotor model that are globally defined on the nonlinear configuration manifold.




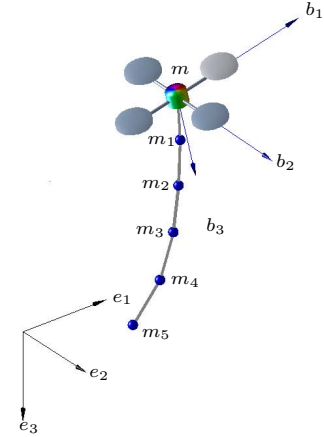
\begin{figure}
\centerline{
	\setlength{\unitlength}{0.09\columnwidth}\scriptsize
\begin{picture}(5,7.5)(0,0)
\put(0,0){\includegraphics[width=0.45\columnwidth]{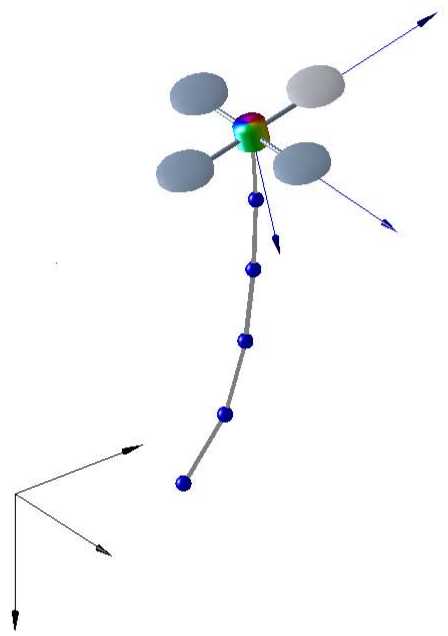}}
\put(2.7,6){\shortstack[c]{$m$}}
\put(2.3,4.8){\shortstack[c]{$m_{1}$}}
\put(2.2,4.05){\shortstack[c]{$m_{2}$}}
\put(2.1,3.25){\shortstack[c]{$m_{3}$}}
\put(2.7,2.4){\shortstack[c]{$m_{4}$}}
\put(2.2,1.5){\shortstack[c]{$m_{5}$}}
\put(1.3,0.8){\shortstack[c]{$e_{2}$}}
\put(1.7,2.1){\shortstack[c]{$e_{1}$}}
\put(0.1,-0.1){\shortstack[c]{$e_{3}$}}
\put(5,7){\shortstack[c]{$b_{1}$}}
\put(4.5,4.4){\shortstack[c]{$b_{2}$}}
\put(3.3,3.3){\shortstack[c]{$b_{3}$}}
\end{picture}
}
\caption{Quadrotor UAV with a cable-suspended load. Cable is modeled as a serial connection of arbitrary number of links (only 5 are illustrated).}\label{fig:Quad}
\end{figure}

The second part of this paper deals with nonlinear control system development. Quadrotor UAV is underactuated as the direction of the total thrust is always fixed relative to its body. By utilizing geometric control systems for quadrotor~\cite{LeeLeoPICDC10,LeeLeoAJC12,GooLeePECC13}, we show that the hanging equilibrium of the links can be asymptotically stabilized while translating the quadrotor to a desired position. In contrast to existing papers where the force and the moment exerted by the payload to the quadrotor are considered as disturbances, the control systems proposed in this paper explicitly consider the coupling effects between the cable/load dynamics and the quadrotor dynamics.

Another distinct feature is that the equations of motion and the control systems are developed directly on the nonlinear configuration manifold in a coordinate-free fashion. This yields remarkably compact expressions for the dynamic model and controllers, compared with local coordinates that often require symbolic computational tools due to complexity of multibody systems. Furthermore, singularities of local parameterization are completely avoided to generate agile maneuvers in a uniform way.



This paper is organized as follows. A dynamic model and control systems are presented at Sections \ref{sec:DM} through \ref{sec:CS}. The desirable properties of the proposed control system are illustrated by a numerical example at Section \ref{sec:NE} and preliminary experimental results are given at Section \ref{sec:ER}. 


\section{Dynamic Model of a Quadrotor with a Flexible Cable}\label{sec:DM}


Consider a quadrotor UAV with a payload that is connected via a chain of $n$ links, as illustrated at Figure \ref{fig:Quad}. The inertial frame is defined by the unit vectors $e_{1}=[1;0;0]$, $e_{2}=[0;1;0]$, and $e_{3}=[0;0;1]\in \Re^{3}$, and the third axis $e_{3}$ corresponds to the direction of gravity. Define a body-fixed frame $\{\vec{b}_{1},\vec{b}_{2},\vec{b}_{3}\}$ whose origin is located at the center of mass of the quadrotor, and its third axis $\vec b_3$ is aligned to the axis of symmetry. 

The location of the mass center, and the attitude of the quadrotor are denoted by $x\in\Re^3$ and $R\in\SO$, respectively, where the special orthogonal group is $\SO=\{R\in\Re^{3\times 3}\,|\, R^T R=I_{3\times 3},\;\mathrm{det}[R]=1\}$. A rotation matrix represents the linear transformation of a representation of a vector from the body-fixed frame to the inertial frame. 

The dynamic model of the quadrotor is identical to~\cite{LeeLeoPICDC10}. The mass and the inertia matrix of the quadrotor are denoted by $m\in\Re$ and $J\in\Re^{3\times 3}$, respectively. The quadrotor can generate a thrust $-fRe_3\in\Re^3$ with respect to the inertial frame, where $f\in\Re$ is the total thrust magnitude. It also generates a moment $M\in\Re^3$ with respect to its body-fixed frame. The pair $(f,M)$ is considered as control input of the quadrotor. 

Let $q_i\in\Sph^2$ be the unit-vector representing the direction of the $i$-th link, measured  from the quadrotor toward the payload, where the two-sphere is the manifold of unit-vectors in $\Re^3$, i.e., $\Sph^2=\{q\in\Re^3\,|\, \|q\|=1\}$. For simplicity, we assume that the mass of each link is concentrated at the outboard end of the link, and the point where the first link is attached to the quadrotor corresponds to the mass center of the quadrotor. The mass and the length of the $i$-th link are defined by $m_i$ and $l_i\in\Re$, respectively. Thus, the mass of the payload corresponds to $m_n$. The corresponding configuration manifold of this system is $\SO\times\Re^3\times (\Sph^2)^n$.

Next, we define the kinematics equations. Let $\Omega\in\Re^3$ be the angular velocity of the quadrotor represented with respect to the body-fixed frame, and let $\omega_i\in\Re^3$ be the angular velocity of the $i$-th link represented with respect to the inertial frame. The angular velocity is normal to the direction of the link, i.e., $q_i\cdot\omega_i=0$. The kinematics equations are given by
\begin{align}
\dot R & = R\hat\Omega,\label{eqn:Rdot}\\
\dot{q}_{i} & =\omega_{i}\times q_{i}=\hat{\omega}_{i}q_{i},\label{eqn:qidot}
\end{align}
where the hat map $\hat\cdot:\Re^3\rightarrow\so$ is defined by the condition that $\hat x y =x\times y$ for any $x,y\in\Re^3$, and it transforms a vector in $\Re^{3}$ to a $3\times 3$ skew-symmetric matrix. More explicitly, it is given by
\begin{align}\label{LinEOM}
\hat{x}=
\begin{bmatrix}
0&-x_{3}&x_{2}\\
x_{3}&0&-x_{1}\\
-x_{2}&x_{1}&0
\end{bmatrix}
\end{align}
for $x=[x_{1},x_{2},x_{3}]^{T}\in \Re^{3}$. 
The inverse of the hat map is denoted by the \textit{vee} map $\vee:\so\rightarrow\Re^3$. 

Throughout this paper, the 2-norm of a matrix $A$ is denoted by $\|A\|$, and the dot product is denoted by $x \cdot y = x^Ty$.


\subsection{Lagrangian}

We derive the equations of motion according to Lagrangian mechanics. The kinetic energy of the quadrotor is given by
\begin{align}
T_Q = \frac{1}{2}m\|\dot x\|^2 + \frac{1}{2} \Omega\cdot J\Omega.\label{eqn:TQ}
\end{align}
Let $x_i\in\Re^3$ be the location of $m_i$ in the inertial frame. It can be written as
\begin{align}\label{posvec}
x_{i}=x+\sum^{i}_{a=1}{l_{a}q_{a}}.
\end{align}
Then, the kinetic energy of the links is given by
\begin{align}
T_L & = \frac{1}{2} \sum_{i=1}^n m_i \|\dot x+\sum^{i}_{a=1}{l_{a}\dot q_{a}}\|^2\nonumber\\
& = \frac{1}{2}\sum_{i=1}^n m_i \|\dot x\| + \dot x\cdot \sum_{i=1}^n\sum_{a=i}^n m_a l_i \dot q_i
+\frac{1}{2}\sum_{i=1}^n m_i \|\sum_{a=1}^i l_a \dot q_a\|^2.
\label{eqn:TL}
\end{align}
From \refeqn{TQ} and \refeqn{TL}, the total kinetic energy can be written as
\begin{align}
T & =\frac{1}{2}M_{00}\|\dot{x}\|^{2}+\dot{x}\cdot\sum^{n}_{i=1}{M_{0i}\dot{q}_{i}}+\frac{1}{2}\sum^{n}_{i,j=1}{M_{ij}\dot{q}_{i}\cdot\dot{q}_{j}}\nonumber\\
&\quad +\frac{1}{2}\Omega^{T}J\Omega,\label{eqn:KE}
\end{align}
where the inertia values $M_{00},M_{0i},M_{ij}\in\Re$ are given by
\begin{gather}
M_{00}=m+\sum_{i=1}^n m_i,\quad M_{0i}=\sum_{a=i}^n m_a l_i,\quad M_{i0}=M_{0i},\nonumber\\
M_{ij}=\braces{\sum_{a=\max\{i,j\}}^n m_a} l_i l_j,\label{eqn:Mij}
\end{gather}
for $1\leq i,j\leq n$. The gravitational potential energy is given by
\begin{align}
V & = -mgx\cdot e_3 - \sum_{i=1}^n m_i g x_i\cdot e_3\nonumber\\
& = -\sum^{n}_{i=1}\sum^{n}_{a=i}m_{a}gl_{i}e_{3}\cdot q_{i}-M_{00}ge_{3}\cdot x,\label{eqn:PE}
\end{align}
From \refeqn{KE} and \refeqn{PE}, the Lagrangian is $L=T-V$.

\subsection{Euler-Lagrange equations}
Coordinate-free form of Lagrangian mechanics on the two-sphere $\Sph^2$ and the special orthogonal group $\SO$ for various multibody systems has been studied in~\cite{Lee08,LeeLeoIJNME08}. The key idea is representing the infinitesimal variation of $R\in\SO$ in terms of the exponential map:
\begin{align}
\delta R = \frac{d}{d\epsilon}\bigg|_{\epsilon = 0} \exp R(\epsilon \hat\eta) = R\hat\eta,\label{eqn:delR}
\end{align}
for $\eta\in\Re^3$. The corresponding variation of the angular velocity is given by $\delta\Omega=\dot\eta+\Omega\times\eta$. Similarly, the infinitesimal variation of $q_i\in\Sph^2$ is given by
\begin{align}
\delta q_i = \xi_i\times q_i,\label{eqn:delqi}
\end{align}
for $\xi_i\in\Re^3$ satisfying $\xi_i\cdot q_i=0$. This lies in the tangent space as it is perpendicular to $q_{i}$. Using these, the Euler-Lagrange equations can be derived as follows.
\begin{prop}\label{prop:FDM}
Consider a quadrotor with a cable-suspended payload whose Lagrangian is given by \refeqn{KE} and \refeqn{PE}. The Euler-Lagrange equations on $\Re^3\times\SO\times(\Sph^2)^n$ are as follows:
\begin{gather}
M_{00}\ddot{x}+\sum^{n}_{i=1}{M_{0i}\ddot{q}_{i}}=-fRe_{3}+M_{00}ge_{3},\label{eqn:xddot}\\
M_{ii}\ddot q_i  -\hat q_i^2 (M_{i0}\ddot x + \sum_{\substack{j=1\\j\neq i}}^n M_{ij}\ddot q_j)\nonumber\\
=- M_{ii}\|\dot q_i\|^2 q_i-\sum_{a=i}^n m_a gl_i\hat q_i^2 e_3,\label{eqn:qddot}\\
J\dot{\Omega}+\hat{\Omega}J\Omega=M,\label{eqn:Wdot}
\end{gather}
where $M_{ij}$ is defined at \refeqn{Mij}. Equations \refeqn{xddot} and \refeqn{qddot} can be rewritten in a matrix form as follows:
\begin{align}
&\begin{bmatrix}%
    M_{00} & M_{01} & M_{02} & \cdots & M_{0n} \\
    -\hat q_1^2 M_{10} & M_{11}I_{3} & -M_{12} \hat q_1^2 & \cdots & -M_{1n}\hat q_1^2\\%
    -\hat q_2^2 M_{20} & -M_{21} \hat q_2^2 & M_{22} I_{3} & \cdots & -M_{2n} \hat q_2^2\\%
    \vdots & \vdots & \vdots & & \vdots\\
    -\hat q_n^2 M_{n0} & -M_{n1} \hat q_n^2 & -M_{n2}\hat q_n^2 & \cdots & M_{nn} I_{3}
    \end{bmatrix}%
    \begin{bmatrix}
    \ddot x \\ \ddot q_1 \\ \ddot q_2 \\ \vdots \\ \ddot q_n
    \end{bmatrix}\nonumber\\
 &\qquad\quad=   \begin{bmatrix}
    -fRe_3 +M_{00}ge_3\\
    -\|\dot q_1\|^2M_{11} q_1 -\sum_{a=1}^n m_a gl_1\hat q_1^2 e_3\\
    -\|\dot q_2\|^2M_{22} q_2 -\sum_{a=2}^n m_a gl_2\hat q_2^2 e_3\\
    \vdots\\
    -\|\dot q_n\|^2M_{nn} q_n - m_n gl_n\hat q_n^2 e_3
    \end{bmatrix}.\label{eqn:ELm}
\end{align}
Or equivalently, it can be written in terms of the angular velocities as
\begin{gather}
\begin{bmatrix}%
    M_{00} & -M_{01}\hat q_1 & -M_{02}\hat q_2 & \cdots & -M_{0n}\hat q_n\\
    \hat q_1 M_{10} & M_{11}I_{3} & -M_{12} \hat q_1 \hat q_2 & \cdots & -M_{1n}\hat q_1 \hat q_n\\%
    \hat q_2 M_{20} &-M_{21} \hat q_2\hat q_1 & M_{22} I_{3} & \cdots & -M_{2n} \hat q_2 \hat q_n\\%
    \vdots & \vdots & \vdots & & \vdots\\
    \hat q_n M_{n0} &-M_{n1} \hat q_n \hat q_1 & -M_{n2}\hat q_n \hat q_2 & \cdots & M_{nn} I_{3}
    \end{bmatrix}%
    \begin{bmatrix}
    \ddot x \\ \dot \omega_1 \\ \dot \omega_2 \\ \vdots \\ \dot \omega_n
    \end{bmatrix}\nonumber\\
    =
    \begin{bmatrix}
    \sum_{j=1}^n M_{0j} \|\omega_j\|^2 q_j-fRe_3+M_{00}ge_3\\
    \sum_{j=2}^n M_{1j}\|\omega_j\|^2\hat q_1 q_j +\sum_{a=1}^n m_a gl_1\hat q_1 e_3\\
    \sum_{j=1,j\neq 2}^n M_{2j}\|\omega_j\|^2\hat q_2 q_j +\sum_{a=2}^n m_a gl_2\hat q_2 e_3\\
    \vdots\\
    \sum_{j=1}^{n-1} M_{nj}\|\omega_j\|^2\hat q_n q_j + m_n gl_n\hat q_n e_3\\
    \end{bmatrix},\label{eqn:ELwm}\\
\dot q_i = \omega_i\times q_i.\label{eqn:ELwm2}
\end{gather}

\end{prop}
\begin{proof}
See Appendix A.
\end{proof}
These provide a coordinate-free form of the equations of motion for the presented quadrotor UAV that is uniformly defined for any number of links $n$, and that is globally defined on the nonlinear configuration manifold. Compared with equations of motion derived in terms of local coordinates, such as Euler-angles, these provide a compact form of equations that are suitable for control system design.

\section{Control System Design for a Simplified Dynamic Model}

\subsection{Control Problem Formulation}

Let $x_d\in\Re^3$ be a fixed desired location of the quadrotor UAV. Assuming that all of the links are pointing downward, i.e., $q_i=e_3$, the resulting location of the payload is given by $x_n=x_d +\sum_{i=1}^n l_i e_3$. We wish to design the control force $f$ and the control moment $M$ such that this hanging equilibrium configuration at the desired location becomes asymptotically stable. 

\subsection{Simplified Dynamic Model}

For the given equations of motion \refeqn{xddot} for $x$, the control force is given by $-fRe_3$. This implies that the total thrust magnitude $f$ can be arbitrarily chosen, but the direction of the thrust vector is always along the third body-fixed axis. Also, the rotational attitude dynamics of the quadrotor is not affected by the translational dynamics of the quadrotor or the dynamics of links.

In this section, we replace the term  $-fRe_3$ of \refeqn{xddot} by a fictitious control input $u\in\Re^3$, and design an expression for $u$ to asymptotically stabilize the desired equilibrium. This is equivalent to assuming that the attitude of the quadrotor can be instantaneously changed. The effects of the attitude dynamics are studied at the next section.

In short, the equations of motion for the simplified dynamic model considered in the section are given by
\begin{align}
M_{00}\ddot{x}+\sum^{n}_{i=1}{M_{0i}\ddot{q}_{i}}=u+M_{00}ge_{3},\label{eqn:xddot_sim}
\end{align}
and \refeqn{qddot}.

\subsection{Linear Control System}\label{sec:LCS}
The fictitious control input is designed from the linearized dynamics about the desired hanging equilibrium. The variation of $x$ and $u$ are given by
\begin{align}
\delta x = x - x_d,\quad \delta u = u - M_{00}g e_3.\label{eqn:delxLin}
\end{align}
From \refeqn{delqi}, the variation of $q_i$ from the equilibrium can be written as
\begin{align}
\delta q_i = \xi_i\times e_3,\label{eqn:delqLin}
\end{align}
where $\xi_i\in\Re^3$ with $\xi_i\cdot e_3=0$. The variation of $\omega_i$ is given by $\delta\omega\in\Re^3$ with $\delta\omega_i \cdot e_3=0$. Therefore, the third element of each of $\xi_i$ and $\delta\omega_i$ for any equilibrium configuration is zero, and they are omitted in the following linearized equation, i.e., the state vector of the linearized equation is composed of $C^T\xi_i\in\Re^2$, where $C=[e_1,e_2]\in\Re^{3\times 2}$. 

\newcommand{\Mb}{\mathbf{M}}
\newcommand{\Kb}{\mathbf{K}}
\newcommand{\Bb}{\mathbf{B}}
\newcommand{\xb}{\mathbf{x}}
\newcommand{\ub}{\mathbf{u}}
\newcommand{\vb}{\mathbf{v}}
\newcommand{\Gb}{\mathbf{G}}

\begin{prop}\label{prop:FDMM}
The linearized equations of the simplified dynamic model \refeqn{xddot_sim} and \refeqn{qddot} can be written as follows:
\begin{gather}
\Mb\ddot \xb  + \Gb\xb = \Bb \delta u,\label{eqn:Lin}
\end{gather}
or equivalently 
\begin{align*}
\begin{bmatrix} \Mb_{xx} & \Mb_{xq}\\ \Mb_{qx} & \Mb_{qq} \end{bmatrix}
\begin{bmatrix}  \delta\ddot x \\ \ddot \xb_q\end{bmatrix}
+
\begin{bmatrix} 0_3 & 0_{3\times 2n}\\ 0_{2n\times 3} & \Gb_{qq}\end{bmatrix}
\begin{bmatrix}  \delta x \\ \xb_q\end{bmatrix}
=
\begin{bmatrix} I_3 \\ 0_{2n\times 3}\end{bmatrix}
\delta u,
\end{align*}
where the corresponding sub-matrices are defined as
\begin{align*}
\xb_q & = [C^T \xi_1;\,\ldots\,;\,C^T \xi_n],\\
\Mb_{xx} &= M_{00}I_{3},\\
\Mb_{xq} &= \begin{bmatrix}
-M_{01}\hat e_3C & -M_{02}\hat e_3C & \cdots & -M_{0n}\hat e_3C
\end{bmatrix},\\
\Mb_{qx} & = \Mb_{xq}^T,\\
\Mb_{qq} &=
    \begin{bmatrix}%
M_{11}I_{2} & M_{12} I_2 & \cdots & M_{1n}I_2\\%
M_{21} I_2 & M_{22} I_{2} & \cdots & M_{2n}I_2\\%
\vdots & \vdots & & \vdots\\
M_{n1}I_2 & M_{n2}I_2 & \cdots & M_{nn} I_{2}
    \end{bmatrix},\\
\Gb_{qq} & = \mathrm{diag}[\sum_{a=1}^n m_a gl_1 I_2,\, \ldots, m_ngl_nI_2].
\end{align*}
\end{prop}
\begin{proof}
See Appendix A.
\end{proof}

For the linearized dynamics \refeqn{Lin}, the following control system is chosen:
\begin{align}
\delta u & = -k_{x}\delta{x}-k_{\dot{x}}\delta\dot{x}-\sum_{a=1}^{n}{k_{q_{i}}C^{T}(e_3\times q_{i})}-k_{\omega_{i}}C^{T}\delta\omega_{i}\nonumber\\
& = -K_x \xb - K_{\dot x} \dot \xb,\label{eqn:delu}
\end{align}
for controller gains $K_x =[k_xI_3,k_{q_1}I_{3\times 2},\ldots k_{q_n}I_{3\times 2}]\in\Re^{3\times (3+2n)}$ and $K_{\dot x} =[k_{\dot x}I_3,k_{\omega_1}I_{3\times 2},\ldots k_{\omega_n}I_{3\times 2}]\in\Re^{3\times (3+2n)}$. Provided that \refeqn{Lin} is controllable, we can choose the controller gains $K_x,K_{\dot x}$ such that the equilibrium is asymptotically stable for the linearized equation \refeqn{Lin}. Then, the equilibrium becomes asymptotically stable for the nonlinear Euler-Lagrange equation \refeqn{xddot_sim} and \refeqn{qddot}~\cite{Kha96}.

%

\section{Controller Design for a Quadrotor with a Flexible Cable}\label{sec:CS}

The control system designed in the previous section is generalized to the full dynamic model that includes the attitude dynamics. The central idea is that the attitude $R$ of the quadrotor is controlled such that its total thrust direction $-Re_3$, that corresponds to the third body-fixed axis, asymptotically follows the direction of the fictitious control input $u$. By choosing the total thrust magnitude properly, we can guarantee  asymptotical stability for the full dynamic model. 


\subsection{Controller Design}
Let $A\in\Re^3$ be the ideal total thrust of the quadrotor system that asymptotically stabilize the desired equilibrium. From \refeqn{delu}, we have 
\begin{align}
A= M_{00}ge_3 + \delta u = -K_{x} \xb-K_{\dot{x}}\dot\xb + M_{00}ge_3.\label{eqn:A}
\end{align}
The desired direction of the third body-fixed axis is given by
\begin{align}
b_{3_c} = - \frac{A}{\|A\|}.\label{eqn:b3c}
\end{align}
This provides a two-dimensional constraint for the desired attitude of quadrotor, and there is additional one-dimensional degree of freedom that corresponds to rotation about the third body-fixed axis, i.e., yaw angle. A desired direction of the first body-fixed axis, namely $b_{1_d}\in\Sph^2$ is introduced to resolve it, and it is projected onto the plane normal to $b_{3_c}$. The desired direction of the second body-fixed axis is chosen to constitute an orthonormal frame. More explicitly, the desired attitude is given by
\begin{align}
R_c = \bracket{-\frac{\hat b_{3_c}^2 b_{1_d}}{\|\hat b_{3_c}^2 b_{1_d}\|},\;
 \frac{\hat b_{3_c}b_{1_d}}{\|\hat b_{3_c}b_{1_d}\|},\; b_{3_c}},
\end{align}
which is guaranteed to lie in $\SO$ by construction, assuming that $b_{1_d}$ is not parallel to $b_{3_c}$. The desired angular velocity is obtained by the attitude kinematics equation:
\begin{align}
\Omega_c = (R_c^T \dot R_c)^\vee.
\end{align}

Next, we introduce the tracking error variables for the attitude and the angular velocity as follows:
\begin{align}
e_{R}&=\frac{1}{2}(R_{c}^{T}R-R^{T}R_{c})^{\vee},\\
e_{\Omega}&=\Omega-R^{T}R_{c}\Omega_{c}.
\end{align}

The thrust magnitude and the moment vector of quadrotor are chosen as 
\begin{align}
f  & = -A\cdot Re_3,\label{eqn:fi}\\
M & = -\frac{k_R}{\epsilon^2}e_{R} -\frac{k_{\Omega}}{\epsilon}e_{\Omega} +\Omega\times J\Omega\nonumber\\
&\quad -J(\hat\Omega R^T R_{c} \Omega_{c} - R^T R_{c}\dot\Omega_{c}),\label{eqn:Mi}
\end{align}
where $\epsilon,k_R,k_\Omega$ are positive constants.

\begin{prop}\label{prop:FDMMM}
Consider the full dynamics model defined by \refeqn{xddot}, \refeqn{qddot}, \refeqn{Wdot}, \refeqn{Rdot}, and \refeqn{qidot}. Control inputs are designed as \refeqn{fi} and \refeqn{Mi}, where the desired control force is given by \refeqn{A}. Then, there exist $\epsilon^\star > 0$, such that for all $\epsilon<\epsilon^\star$, the
desired equilibrium is exponentially stable.
\end{prop}
\begin{proof}
See Appendix A.
\end{proof}

\section{Numerical Example}\label{sec:NE}

The desirable properties of the proposed control system are illustrated by a numerical example. Properties of a quadrotor are chosen as
\begin{align*}
m=0.5\,\mathrm{kg},\quad J=\mathrm{diag}[0.557,\,0.557,\,1.05]\times 10^{-2}\,\mathrm{kgm^2}.
\end{align*}
Five identical links with $n=5$, $m_i=0.1\,\mathrm{kg}$, and $l_i=0.1\,\mathrm{m}$ are considered. Controller parameters are selected as follows: $k_x=12.8$, $k_v=4.22$, $\frac{k_R}{\epsilon^2}=0.65$, $\frac{k_\Omega}{\epsilon} = 0.11$, $k_q=[11.01,\,6.67,\,1.97,\,0.41,\,0.069]$ and $k_\omega=[0.93,\,0.24,\,0.032,\,0.030,\,0.025]$.

The desired location of the quadrotor is selected as $x_d=0_{3\times 1}$. The initial conditions for the quadrotor are given by
\begin{gather*}
x(0)=[0.6;-0.7;0.2], \ \dot{x}(0)=0_{3\times 1},\\
R(0)=I_{3\times 3},\quad \Omega(0)=0_{3\times 1}.
\end{gather*}
The initial direction of the links are chosen such that the cable is curved along the horizontal direction, as illustrated at Figure \ref{fig:fisrt_sub11}, and the initial angular velocity of each link is chosen as zero. 

\begin{figure}
\centerline{
	\subfigure[Quadrotor position $x$:blue, $x_{d}$:red]{
		\includegraphics[width=0.55\columnwidth]{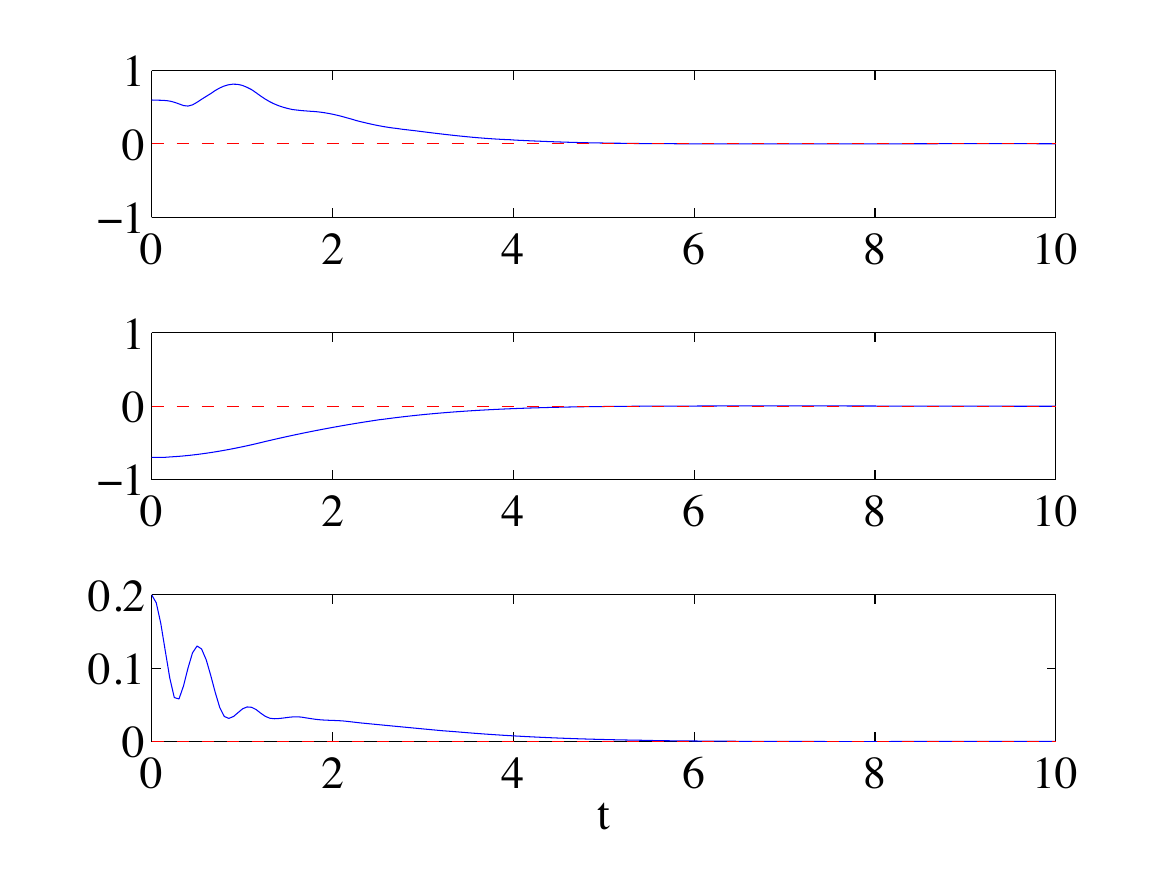}}
	\subfigure[Direction error $e_{q}$ and angular velocity error $e_{\omega}$ for links]{
		\includegraphics[width=0.55\columnwidth]{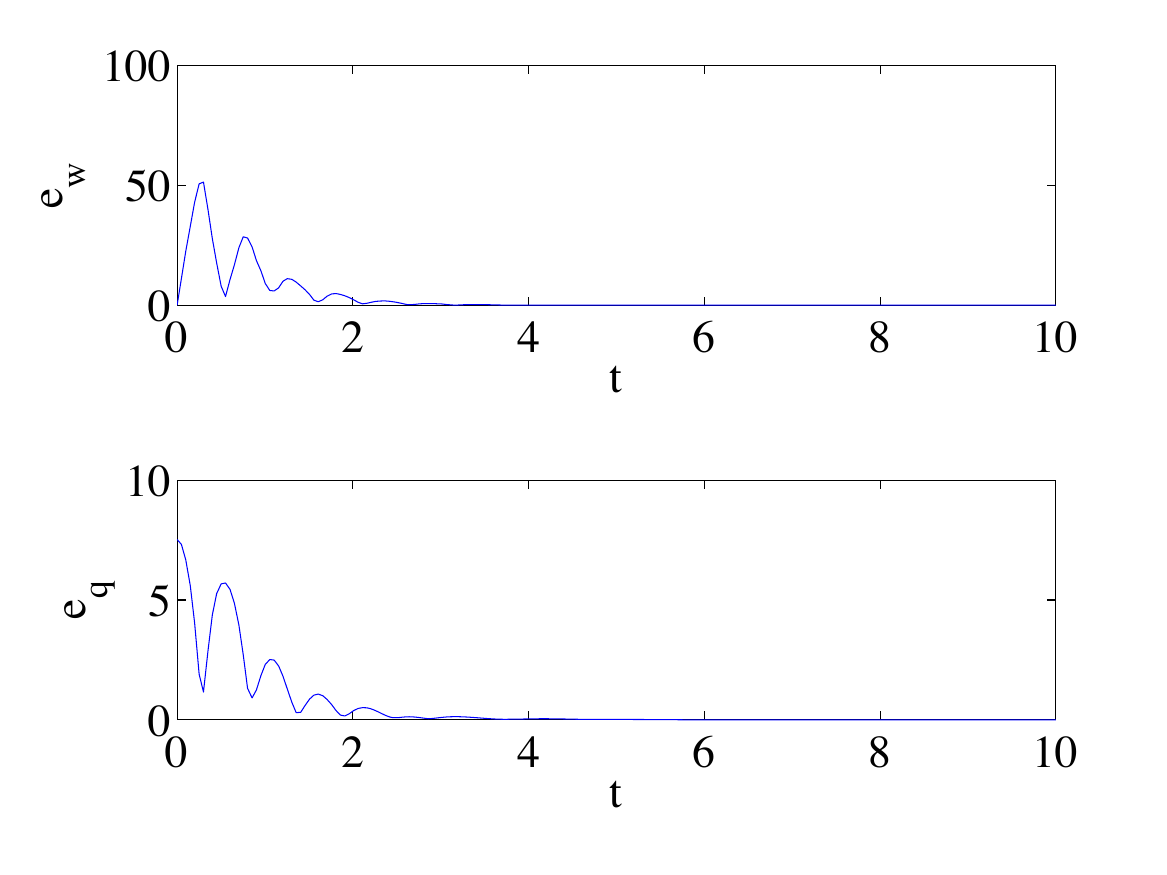}}
}
\centerline{
	\subfigure[Quadrotor angular velocity $\Omega$:blue, $\Omega_{d}$:red]{
		\includegraphics[width=0.55\columnwidth]{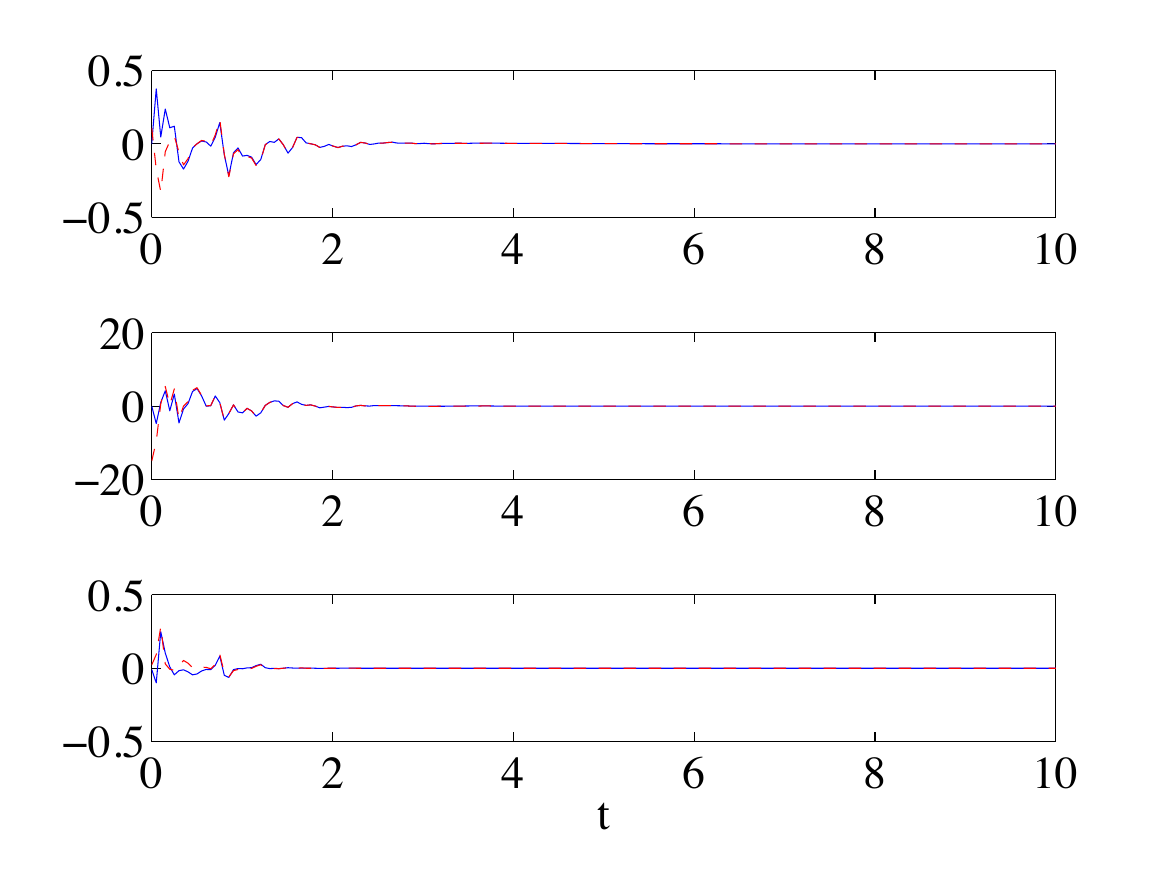}}
	\subfigure[Control force $u$]{
		\includegraphics[width=0.55\columnwidth]{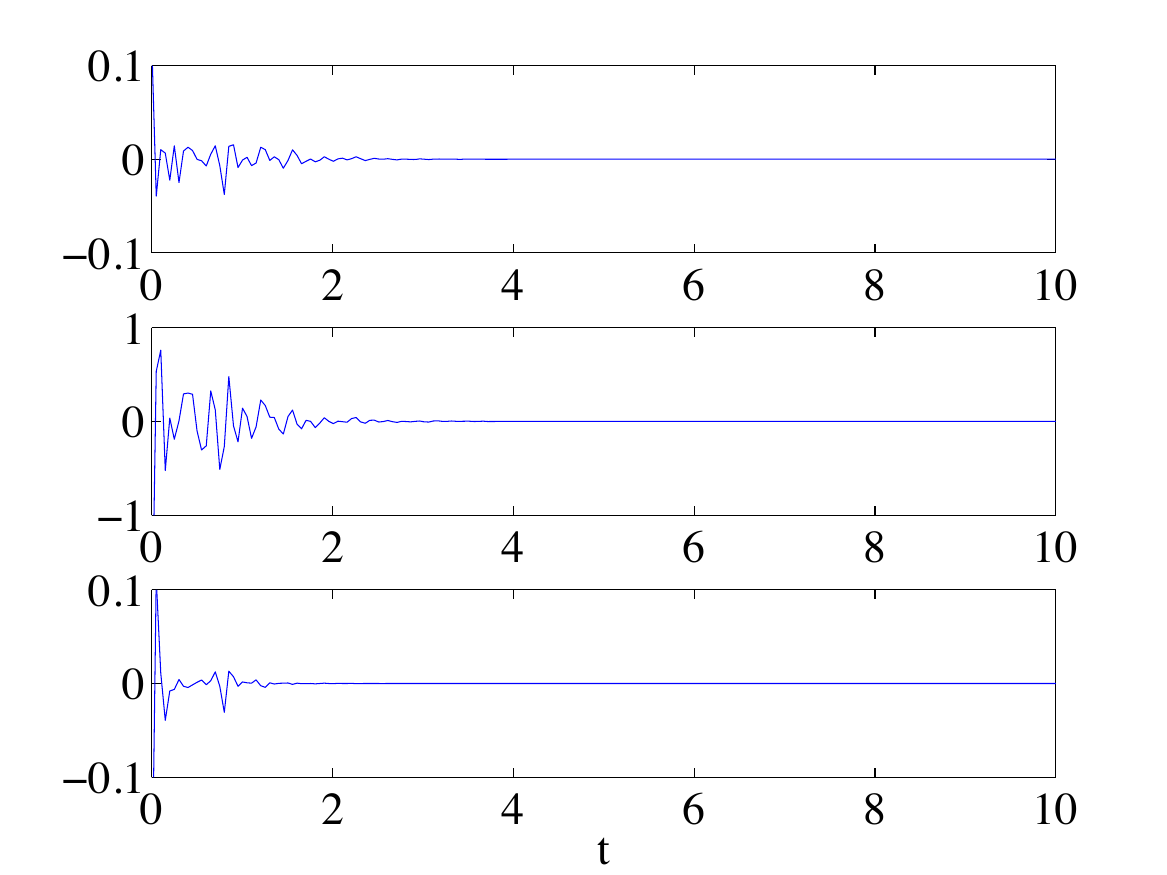}}
}
\caption{Stabilization of a payload connected to a quadrotor with 5 links}\label{fig:simresults}
\end{figure}

We define the following two error functions to the performance of the proposed control system:
\begin{align}
e_{q}=\sum_{i=1}^{n}{\|q_{i}-e_3\|},\quad e_{\omega}=\sum_{i=1}^{n}{\|\omega_{i}\|}
\end{align}
Simulation results are illustrated at Figure \ref{fig:simresults}, where the position $x$ of the quadrotor converges to the desired value $x_d$, while reducing the direction error and the angular velocity error of the link. The corresponding maneuvers of the quadrotor and the links are illustrated by snapshots at Figure \ref{animationsim}. 

\begin{figure}
\centering
\subfigure[$ t =0 $]
{
\includegraphics[width=0.5in]{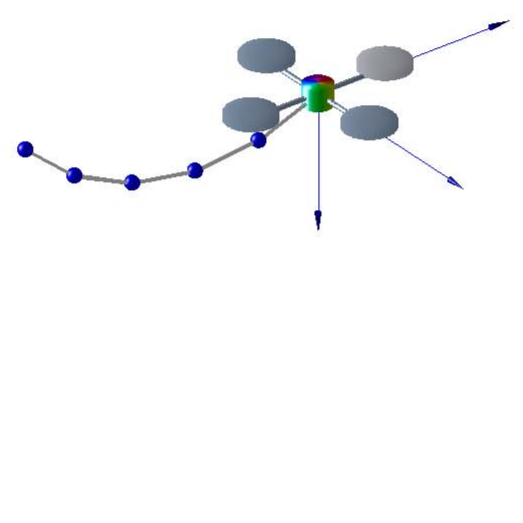}
\label{fig:fisrt_sub11}
}
\subfigure[$ t =0.2 $]
{
\includegraphics[width=0.5in]{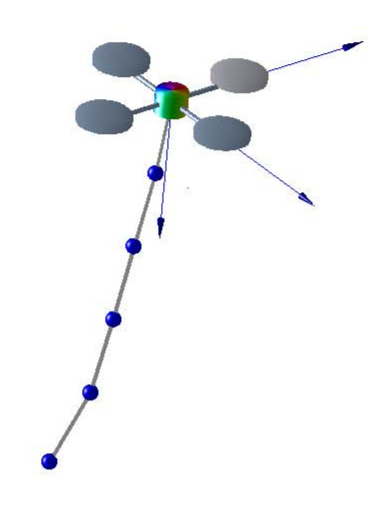}
}
\subfigure[$ t =0.35 $]
{
\includegraphics[width=0.5in]{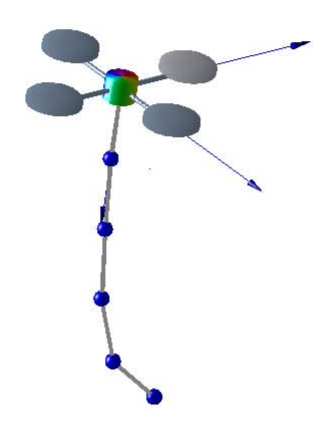}
}
\subfigure[$ t =0.40 $]
{
\includegraphics[width=0.5in]{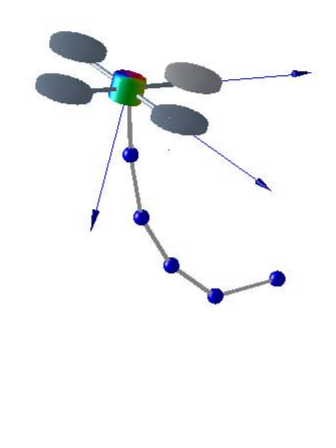}
}
\subfigure[$ t =0.42 $]
{
\includegraphics[width=0.5in]{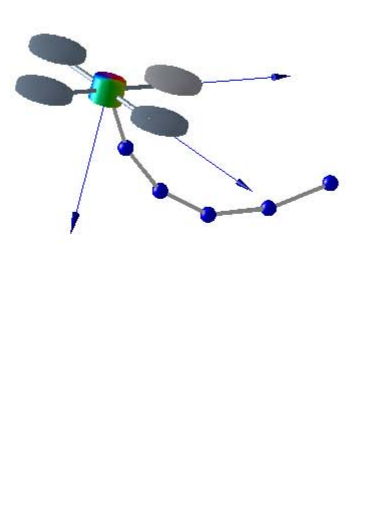}
}\\
\subfigure[$ t =0.45 $]
{
\includegraphics[width=0.5in]{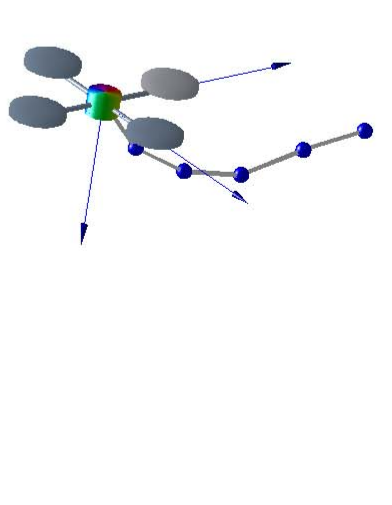}
}
\subfigure[$ t =0.5 $]
{
\includegraphics[width=0.5in]{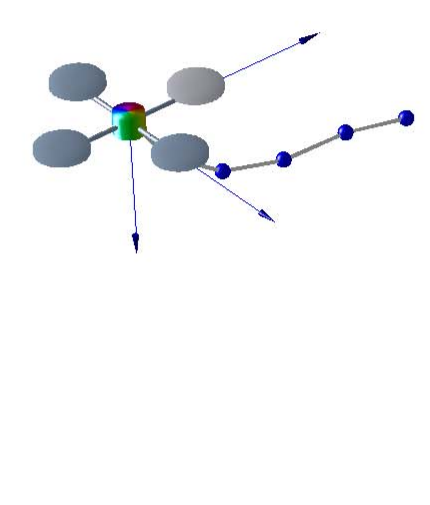}
}
\subfigure[$ t =0.6 $]
{
\includegraphics[width=0.5in]{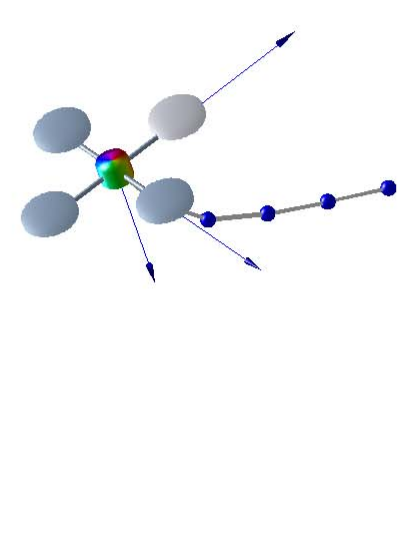}
}
\subfigure[$ t =0.7 $]
{
\includegraphics[width=0.5in]{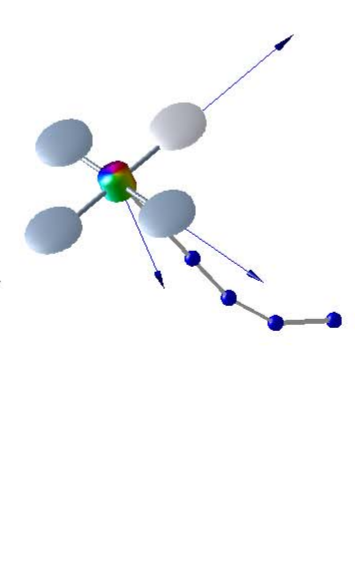}
}
\subfigure[$ t =0.8 $]
{
\includegraphics[width=0.5in]{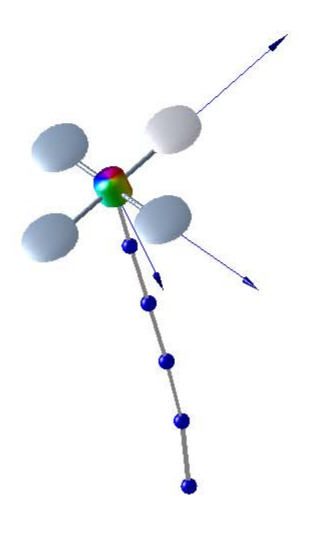}
}\\\subfigure[$ t =0.9 $]
{
\includegraphics[width=0.5in]{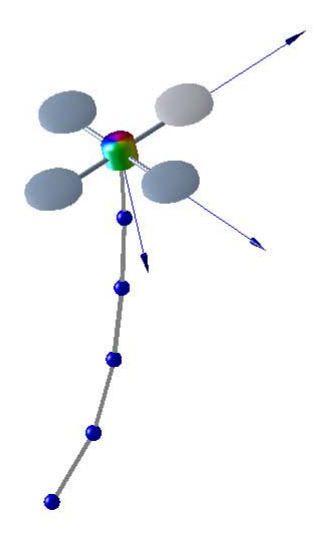}
}
\subfigure[$ t =1.3 $]
{
\includegraphics[width=0.5in]{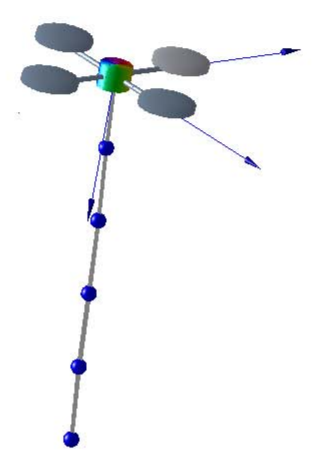}
}
\subfigure[$ t =1.5 $]
{
\includegraphics[width=0.5in]{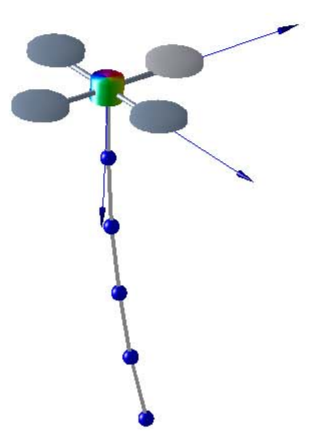}
}
\subfigure[$ t =2.0 $]
{
\includegraphics[width=0.5in]{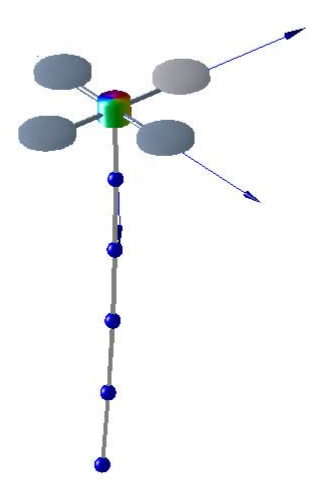}
}
\subfigure[$ t =10.0 $]
{
\includegraphics[width=0.5in]{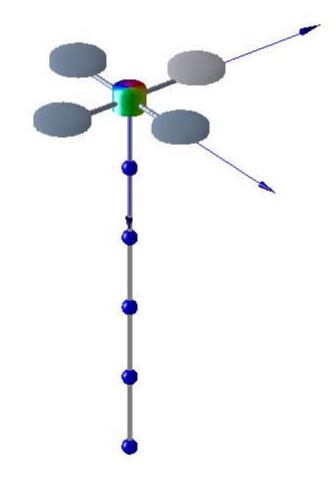}
}
\caption{Snapshots of the controlled maneuver}
\label{animationsim}
\end{figure}

%
%

\section{Preliminary Experimental Results}\label{sec:ER}
Preliminary experimental results are presented. A quadrotor UAV is developed with the following configuration, as illustrated at Figure \ref{fig:QuadHW}:
\begin{itemize}
\item Gumstix Overo computer-in-module (OMAP 600MHz processor), running a non-realtime Linux operating system. It is connected to a ground station via WIFI.
\item Microstrain 3DM-GX3 attitude sensor, connected to Gumstix via UART.
\item BL-CTRL 2.0 motor speed controller, connected to Gumstix via I2C.
\item Roxxy 2827-35 Brushless DC motors.
\item XBee RF module, connected to Gumstix via UART.
\end{itemize}
The mass of the quadrotor is $m=0.755\,\mathrm{kg}$. A payload with mass of $m_1=0.036\ \mathrm{kg}$ is attached to the quadrotor via a cable of length $l_1=0.5\ \mathrm{m}$. For the given preliminary experiments, the cable is modeled by a single link, i.e., $n=1$. The locations of the quadrotor and the payload are measured by a VICON motion capture system, and they are transferred to the onboard computer module via XBee to estimate the full state and compute the control inputs. 

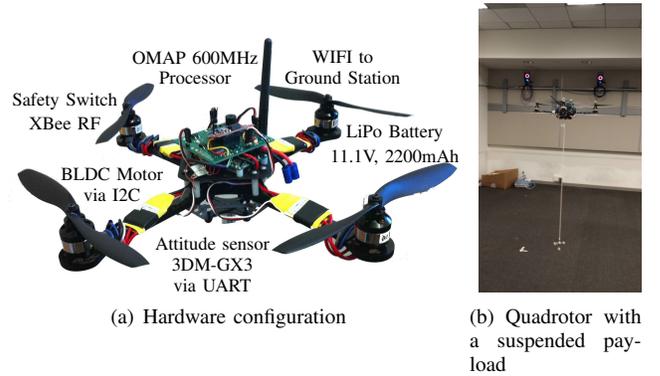
\begin{figure}
\centerline{
	\subfigure[Hardware configuration]{
\setlength{\unitlength}{0.1\columnwidth}\scriptsize
\begin{picture}(7,4)(0,0)
\put(0,0){\includegraphics[width=0.7\columnwidth]{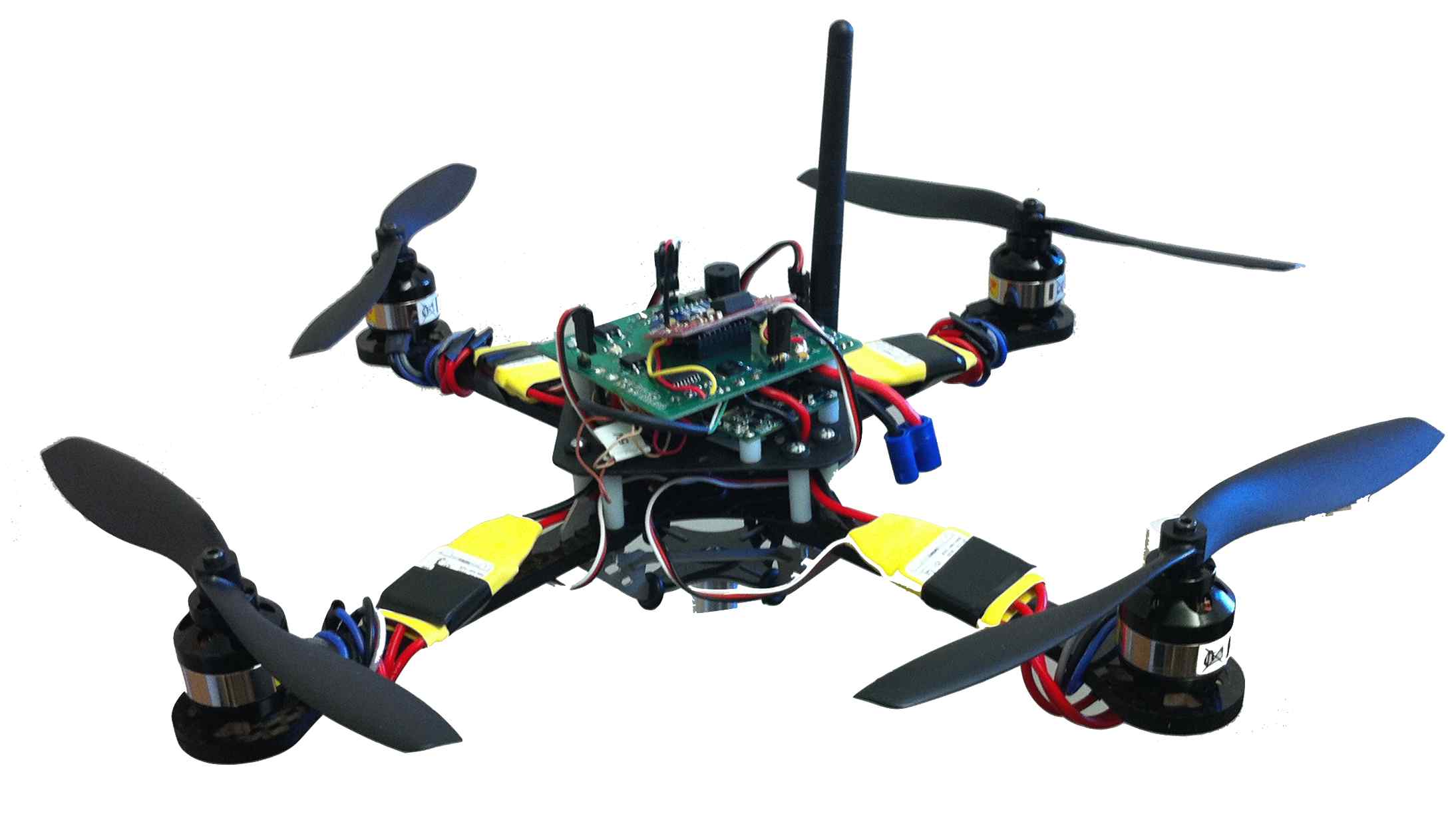}}
\put(1.95,3.2){\shortstack[c]{OMAP 600MHz\\Processor}}
\put(2.3,0){\shortstack[c]{Attitude sensor\\3DM-GX3\\ via UART}}
\put(0.85,1.4){\shortstack[c]{BLDC Motor\\ via I2C}}
\put(0.1,2.5){\shortstack[c]{Safety Switch\\XBee RF}}
\put(4.3,3.2){\shortstack[c]{WIFI to\\Ground Station}}
\put(5,2.0){\shortstack[c]{LiPo Battery\\11.1V, 2200mAh}}
\end{picture}}
	\subfigure[Quadrotor with a suspended payload]{
	\includegraphics[width=0.25\columnwidth]{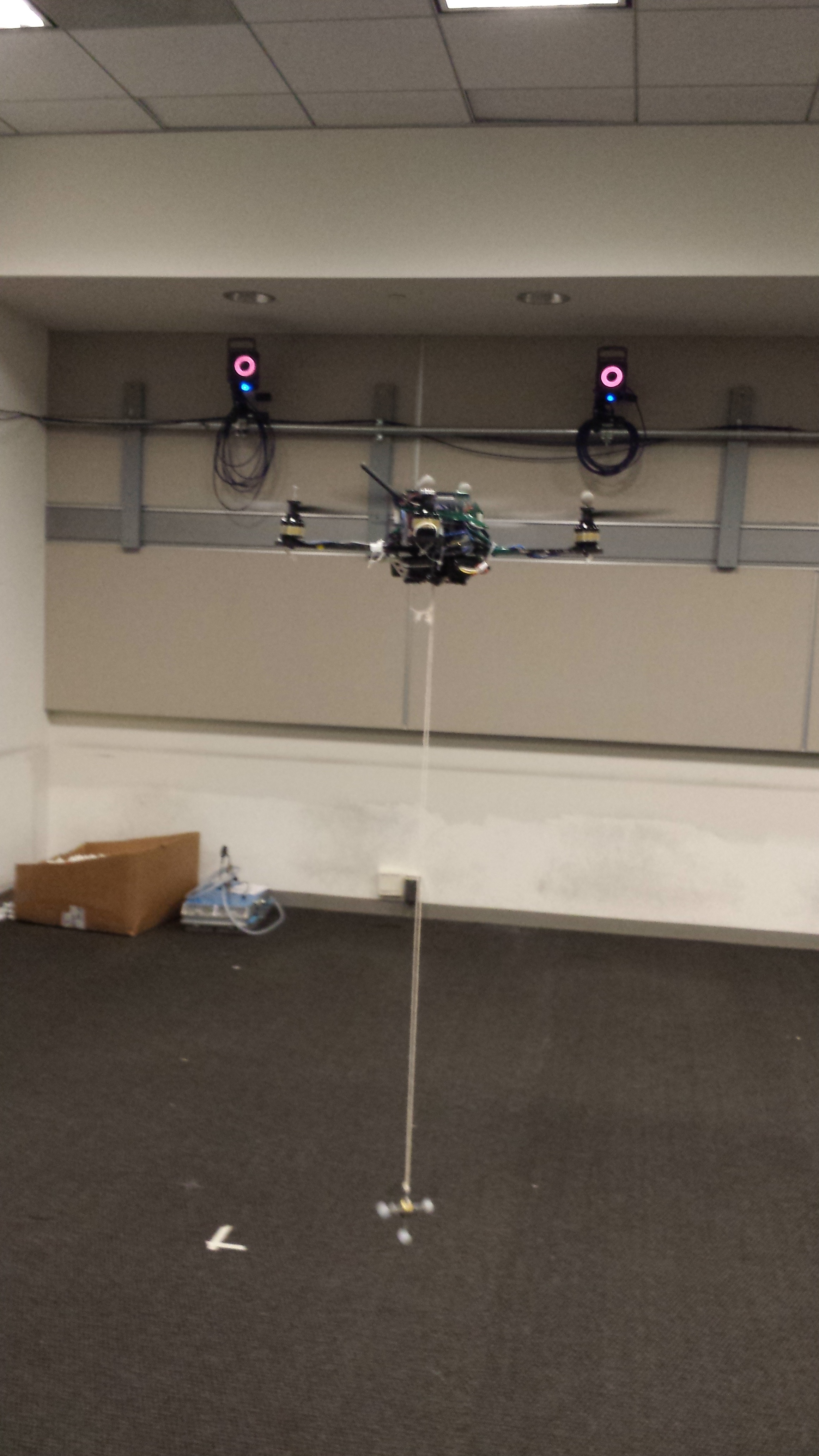}}
}
\caption{Hardware development for a quadrotor UAV}\label{fig:QuadHW}
\end{figure}

Two cases are considered and compared. For the first case, a position control system developed in~\cite{GooLeePECC13}, for quadrotor UAV that does not include the dynamics of the payload and the link, is applied to hover the quadrotor at the desired location, and the second case, the proposed control system is used.

\begin{figure}
\centerline{
\subfigure[Case I: quadrotor position control system~\cite{GooLeePECC13}]
{\includegraphics[width=1\columnwidth]{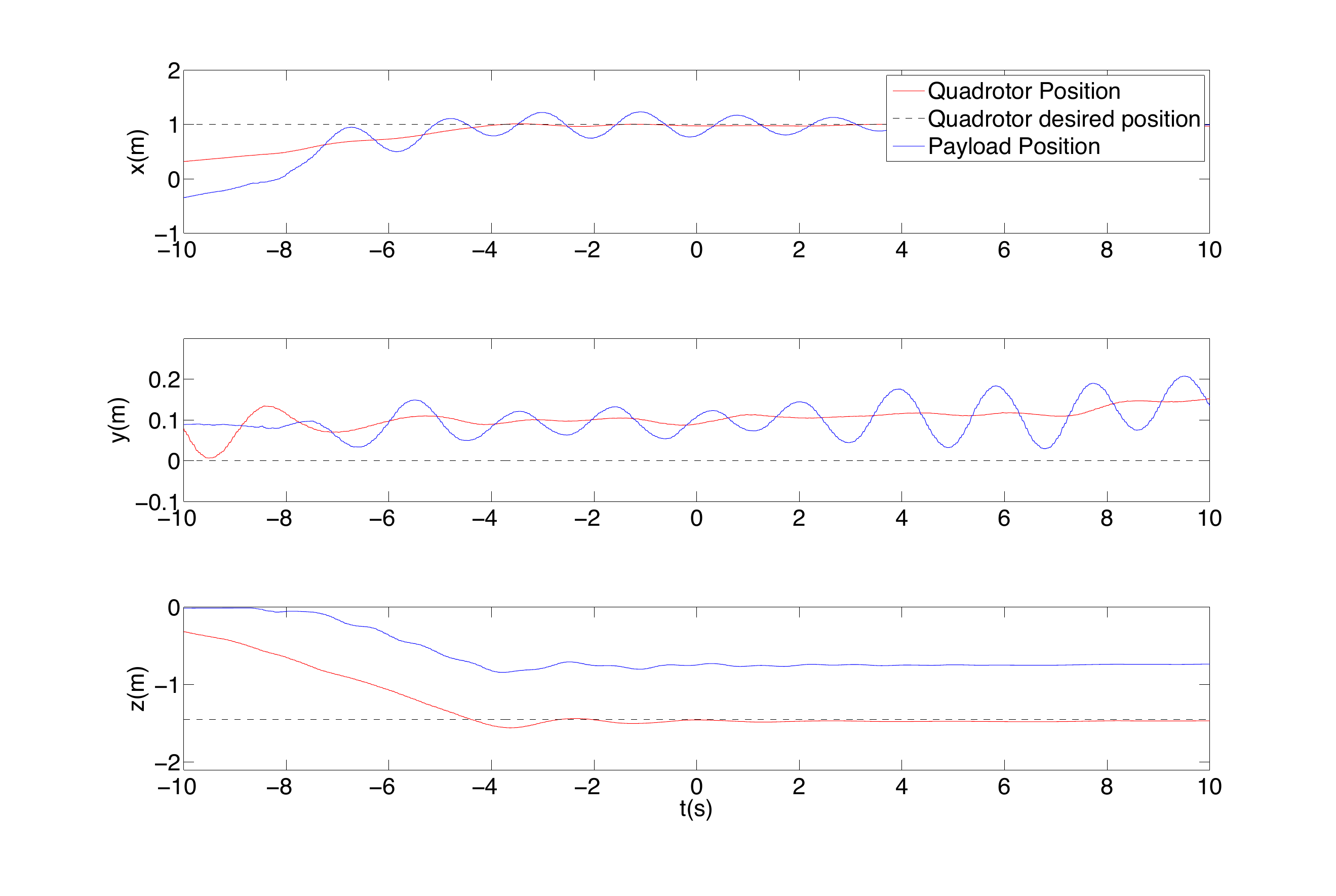}}}
\centerline{
\subfigure[Case II: the proposed control system for quadrotor with a suspended payload]
{\includegraphics[width=1.0\columnwidth]{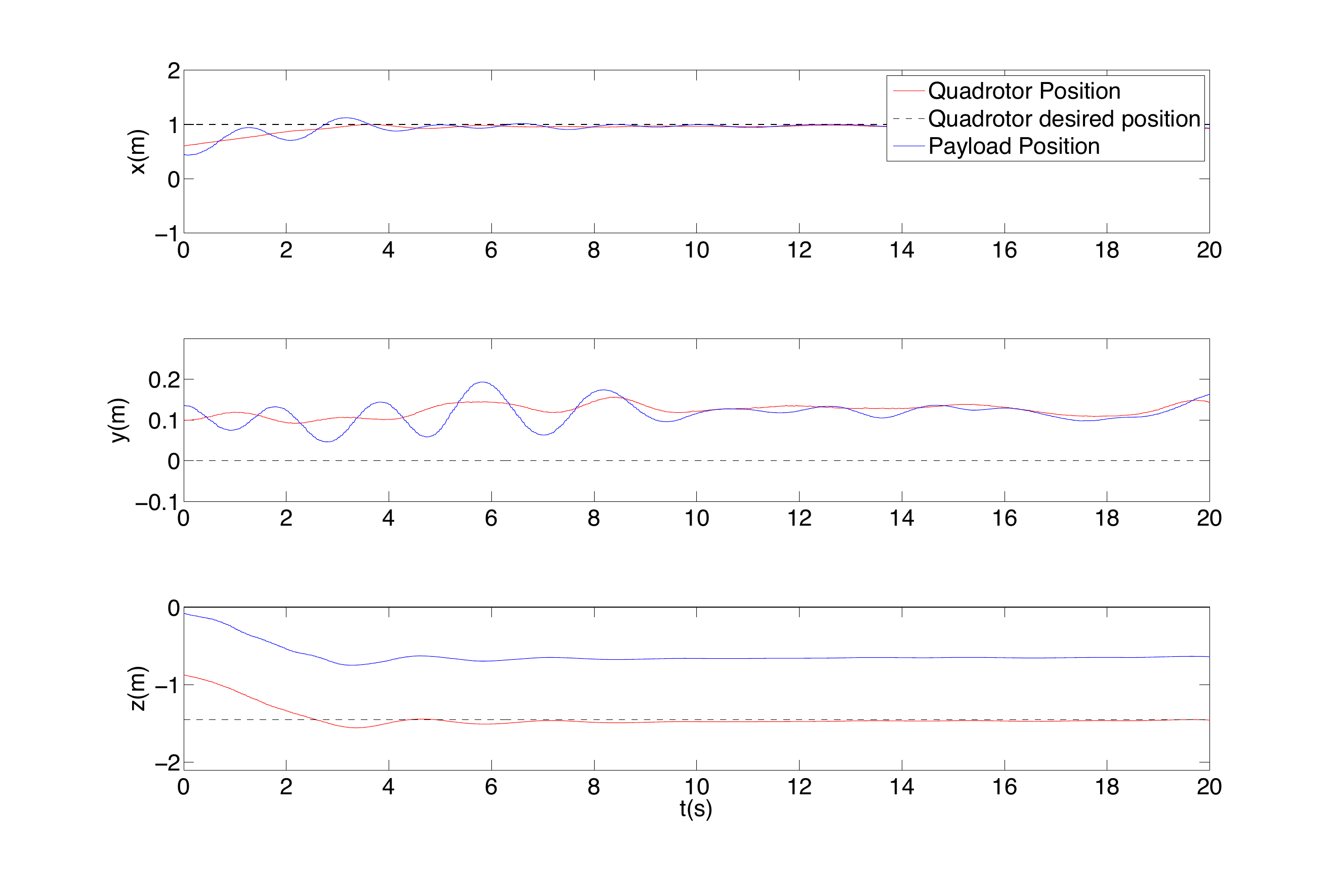}}}
\caption{Preliminary experimental results $x_d$:black, $x$:red, $x+l_1q_1$:blue}
\label{expresultsp}
\end{figure}
Experimental results are shown at Figures \ref{expresultsp} and \ref{expresultsq}. The position of the quadrotor and the payload is compared with the desired position of the quadrotor at Figure~\ref{expresultsp}, and the deflection angle of the link from the vertical direction are illustrated at Figure~\ref{expresultsq}. It is shown that the proposed control system reduces the undesired oscillation of the link effectively, compared with the quadrotor position control system that does not include the dynamics of links and payload.\footnote{A short video file of the experiments is also available at the experiment section of \url{http://fdcl.seas.gwu.edu/}.}

\begin{figure}
\subfigure[Case I: quadrotor position control system~\cite{GooLeePECC13}]
{\includegraphics[width=0.49\columnwidth]{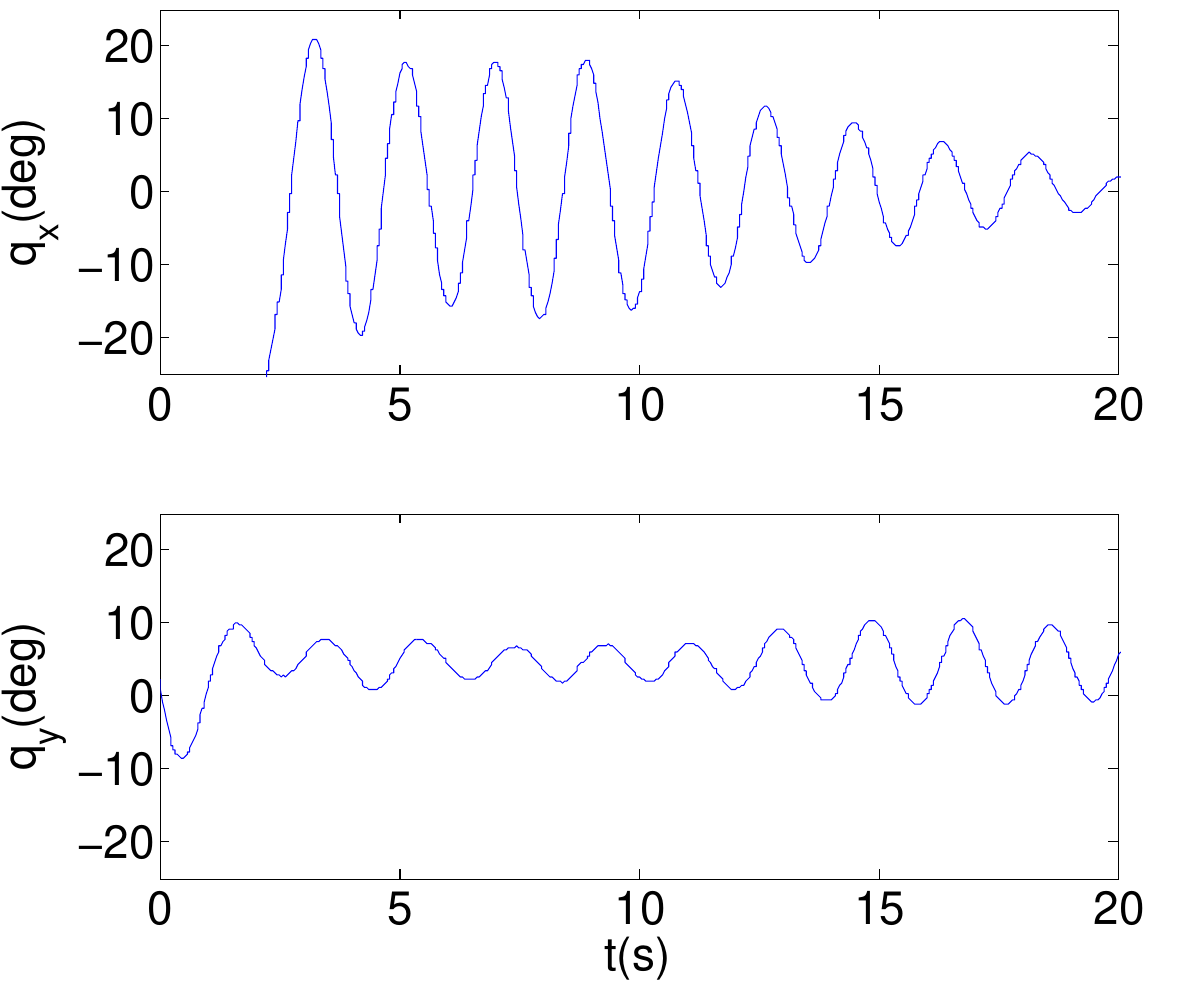}}
\subfigure[Case II: the proposed control system for quadrotor with a suspended payload]
{\includegraphics[width=0.49\columnwidth]{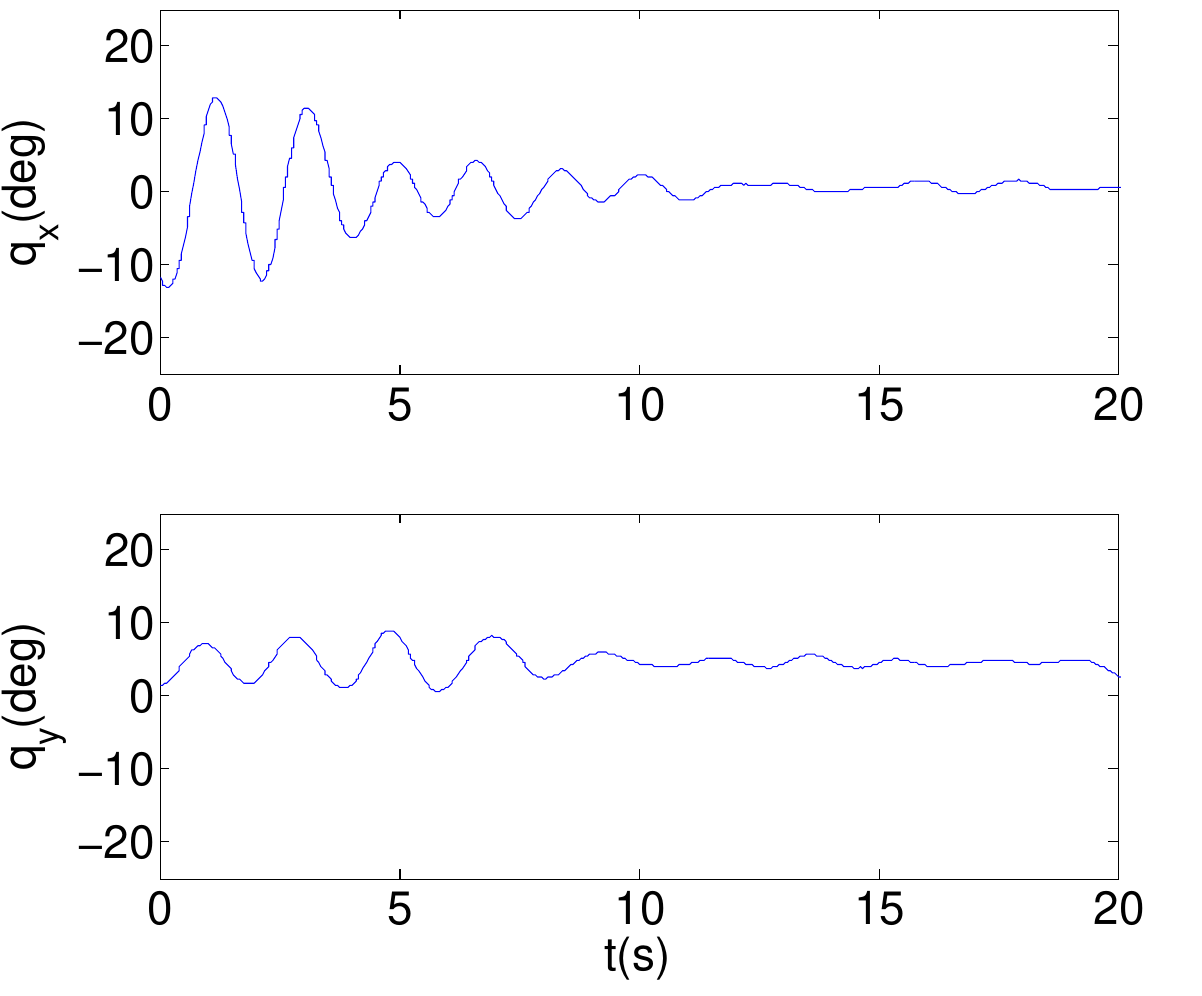}}
\caption{Preliminary experimental results: link deflection angles}
\label{expresultsq}
\end{figure}


\section{Conclusions}
Euler-Lagrange equations have been used for the quadrotor and the chain pendulum to model a flexible cable transporting a load in 3D space. These derivations developed in a remarkably compact form which allow us to choose arbitrary number and any configuration of the links. We developed a geometric nonlinear controller to stabilize the links below the quadrotor in the equilibrium position from an any chosen initial condition. We expanded these derivations in such way that there is no need of using local angle coordinate and this advantageous technique signalize our derivations.

\appendix
\section*{APPENDIX}
\subsection{Proof for Proposition \ref{prop:FDM}}\label{sec:PfFDM}
From \refeqn{KE} and \refeqn{PE}, the Lagrangian is given by
\begin{align}
L&=\frac{1}{2}M_{00}\|\dot{x}\|^{2}+\dot{x}\cdot\sum^{n}_{i=1}{M_{0i}\dot{q}_i}+\frac{1}{2}\sum_{i=1}^{n}{m_{i}\|\sum_{a=1}^{i}{l_{a}\dot{q}_{a}}\|^{2}}\nonumber \\
&\quad+\sum^{n}_{i=1}\sum^{n}_{a=i}m_{a}gl_{i}e_{3}\cdot q_{i}+M_{00}ge_{3}\cdot x+\frac{1}{2}\Omega^{T}J\Omega,
\end{align}
The derivatives of the Lagrangian are given by
\begin{align*}
\D_x L & = M_{00} g e_3,\\
\D_{\dot x} L & = M_{00} \dot x + \sum_{i=1}^n M_{0i}\dot q_i,
\end{align*}
where $\D_x L$ represents the derivative of $L$ with respect to $x$. From the variation of the angular velocity given after \refeqn{delR}, we have
\begin{align}
\D_{\Omega}L\cdot\delta\Omega=J\Omega\cdot(\dot\eta+\Omega\times\eta)
=J\Omega\cdot\dot\eta- \eta\cdot(\Omega\times J\Omega).
\end{align}
Similarly from \refeqn{delqi}, the derivative of the Lagrangian with respect to $q_i$ is given by
\begin{align*}
\D_{q_i} L \cdot \delta q_i = \sum_{a=i}^n m_a gl_ie_3\cdot (\xi_i\times q_i) 
= -\sum_{a=i}^n m_a gl_i\hat e_3 q_i\cdot \xi_i.
\end{align*}
The variation of $\dot q_i$ is given by
\begin{align*}
\delta\dot q_i = \dot \xi_i\times q_i + \xi_i\times q_i.
\end{align*}
Using this, the derivative of the Lagrangian with respect to $\dot q_i$ is given by
\begin{align*}
&\D_{\dot q_i}L\cdot \delta \dot q_i  = (M_{i0}\dot x + \sum_{j=1}^n M_{ij}\dot q_j) \cdot \delta \dot q_i \\
& = (M_{i0}\dot x + \sum_{j=1}^n M_{ij}\dot q_j) \cdot (\dot \xi_i \times q + \xi_i \times \dot q_i)\\
& = \hat q_i (M_{i0}\dot x + \sum_{j=1}^n M_{ij}\dot q_j)\cdot \dot\xi_i +
\hat{\dot q}_i (M_{i0}\dot x + \sum_{j=1}^n M_{ij}\dot q_j)\cdot \xi_i.
\end{align*}

Let $\mathfrak{G}$ be the action integral, i.e., $\mathfrak{G}=\int_{t_0}^{t_f} L\,dt$. From the above expressions for the derivatives of the Lagrangian, the variation of the action integral can be written as
\begin{align*}
\delta \mathfrak{G} &= \int_{t_0}^{t_f} \{M_{00} \dot x + \sum_{i=1}^n M_{0i}\dot q_i\}\cdot \delta \dot x
+M_{00}ge_3\cdot\delta x\\
& +\sum_{i=1}^n \{\hat q_i (M_{i0}\dot x + \sum_{j=1}^n M_{ij}\dot q_j)\}\cdot \dot \xi_i\\
 &+ \sum_{i=1}^n\{\hat{\dot q}_i (M_{i0}\dot x + \sum_{j=1}^n M_{ij}\dot q_j)-\sum_{a=i}^n m_a gl_i\hat e_3 q_i \}\cdot \xi_i\\
 &+ J\Omega\cdot\dot\eta- \eta\cdot(\Omega\times J\Omega)\,dt.
\end{align*}
Integrating by parts and using the fact that variations at the end points vanish, this reduces to
\begin{align*}
\delta \mathfrak{G} &= \int_{t_0}^{t_f} \{M_{00}ge_3 -M_{00} \ddot x - \sum_{i=1}^n M_{0i}\ddot q_i\}\cdot \delta x \\
&+ \sum_{i=1}^n\{-\hat q_i (M_{i0}\ddot x + \sum_{j=1}^n M_{ij}\ddot q_j)-\sum_{a=i}^n m_a gl_i\hat e_3 q_i \}\cdot \xi_i\\
& - \eta\cdot(J\dot\Omega+\Omega\times J\Omega) \,dt.
\end{align*}
According to the Lagrange-d'Alembert principle, the variation of the action integral is equal to the negative of the virtual work done by the external force and moment, namely
\begin{align*}
-\int_{t_0}^{t_f} -fRe_3\cdot\delta x - M\cdot \eta\, dt. 
\end{align*}
Therefore, we obtain \refeqn{xddot} and \refeqn{Wdot}. As $\xi_i$ is perpendicular to $q_i$, we also have
\begin{gather}
-\hat q_i^2 (M_{i0}\ddot x + \sum_{j=1}^n M_{ij}\ddot q_j)+\sum_{a=i}^n m_a gl_i\hat q_i^2 e_3=0.\label{eqn:qddot0}
\end{gather}
Equation \refeqn{qddot0} is rewritten to obtain an explicit expression for $\ddot q_i$. As $q_i\cdot \dot q_i =0$, we have $\dot q_i \cdot \dot q_i +q_i\cdot \ddot q_i=0$. Using this, we have
\begin{align*}
-\hat q_i^2 \ddot q_i = -(q_i\cdot \ddot q_i )q_i + (q_i\cdot q_i)\ddot q_i =(\dot q_i \cdot \dot q_i) q_i + \ddot q_i.
\end{align*}
Substituting this into \refeqn{qddot0}, we obtain \refeqn{qddot}.

This can be slightly rewritten in terms of the angular velocities. Since $\dot q_i = \omega_i\times q_i$ for the angular velocity $\omega_i$ satisfying $q_i\cdot\omega_i=0$, we have
\begin{align*}
    \ddot q_i & = \dot \omega_i \times q_i + \omega_i\times (\omega_i\times q_i) \\
    &= \dot \omega_i \times q_i - \|\omega_i\|^2q_i=-\hat q_i \dot\omega_i - \|\omega_i\|^2q_i.
\end{align*}
Using this and the fact that $\dot\omega_i\cdot q_i=0$, we obtain \refeqn{ELwm}.


The variations of $x,u$ and $q$ are given by \refeqn{delxLin} and \refeqn{delqLin}. From the kinematics equation $\dot q_i=\omega_i\times q_i$, $\delta\dot q_i$ is given by
\begin{align*}
\delta \dot q_i = \dot\xi_i \times e_3 =\delta\omega_i \times e_3 + 0\times (\xi_i\times e_3)=\delta\omega_i \times e_3.
\end{align*}
Since both sides of the above equation is perpendicular to $e_3$, this is equivalent to $e_3\times(\dot\xi_i\times e_3) = e_3\times(\delta\omega_i\times e_3)$, which yields
\begin{gather*}
\dot \xi - (e_3\cdot\dot\xi) e_3 = \delta\omega_i -(e_3\cdot\delta\omega_i)e_3.
\end{gather*}
Since $\xi_i\cdot e_3 =0$, we have $\dot\xi\cdot e_3=0$. As $e_3\cdot\delta\omega_i=0$ from the constraint, we obtain the linearized equation for the kinematics equation:
\begin{align}
\dot\xi_i = \delta\omega_i.\label{eqn:dotxii}
\end{align}
Substituting these into \refeqn{ELwm}, and ignoring the higher order terms, we obtain \refeqn{Lin}.

\subsection{Proof for Proposition \ref{prop:FDMMM}}\label{sec:PfFDMMM}

This proof is based on singular perturbation~\cite{Kha96} and the attitude tracking control system developed in~\cite{LeeLeoPICDC10}. Let $\bar e_{R}=\frac{1}{\epsilon}e_{R}$. The error dynamics for $\bar e_{R},e_{\Omega}$ can be written as
The error dynamics for $e_{R_i},e_{\Omega_i}$ are given by
\begin{align*}
\dot e_{R_i} & = \frac{1}{2}(\tr{R_i^T R_{c_i}}I-R_i^T R_{c_i}) e_{\Omega_i},\\
J_i\dot e_{\Omega_i} & = -\frac{k_R}{\epsilon^2} e_{R_i} -\frac{k_\Omega}{\epsilon} e_{\Omega_i}.
\end{align*}
Define $\bar e_{R_i}=\frac{1}{\epsilon}e_{R_i}$ to rewrite these as the standard form of singular perturbation:
\begin{align*}
\epsilon \dot{\bar e}_{R} & = \frac{1}{2}(\tr{R^T R_{c}}I-R_i^T R_{c}) e_{\Omega},\\
\epsilon \dot e_{\Omega} & = J^{-1}(-k_R \bar e_{R} -k_\Omega e_{\Omega}).
\end{align*}
The right-hand side of the above equations has an isolated root of $(\bar e_{R},e_{\Omega})=(0,0)$, and they correspond to the \textit{boundary-layer} system. And, the origin of the boundary-layer system is exponentially stable according to~\cite[Proposition 1]{LeeLeoPICDC10}.

More explicitly, define a configuration error function on $\SO$ as follows:
\begin{align*}
\Psi_R = \frac{1}{2}\tr{I- R_c^T R}.
\end{align*}
Consider a domain given by $D_R=\{(R,\Omega)\in\SO\times\Re^3\,|\, \Psi_R < \psi_R < 2\}$. Define a Lyapunov function,
\begin{align*}
\mathcal{W} & = \frac{1}{2} e_\Omega\cdot J e_\Omega + \frac{k_R}{\epsilon^2} \Psi_R + \frac{c_3}{\epsilon} e_R\cdot e_\Omega,
\end{align*}
where $c_3$ is a positive constant satisfying
\begin{align*}
c_3 < \min \braces{ \sqrt{k_R\lambda_m(J)},\,\frac{4k_Rk_\Omega\lambda_{m}^2(J)}{k_\Omega^2\lambda_M(J)+4 k_R\lambda_m^2(J)}}.
\end{align*}
We can show that
\begin{align*}
\zeta^T L_1 \zeta \leq \mathcal{W} \leq \zeta^T L_2 \zeta,
\end{align*}
where $\zeta=[\|\bar e_R\|,\, \|e_\Omega\|]\in\Re^2$ and the matrices $L_1,L_2\in\Re^{2\times 2}$ are given by
\begin{align*}
L_1 = \begin{bmatrix} \frac{k_R}{2} & - \frac{c_3}{2} \\ -\frac{c_3}{2} & \frac{\lambda_m(J)}{2}\end{bmatrix},\quad
L_2 = \begin{bmatrix} \frac{k_R}{2-\psi_R} & \frac{c_3}{2} \\ \frac{c_3}{2} & \frac{\lambda_M(J)}{2}\end{bmatrix}.
\end{align*}
The time-derivative of $\mathcal{W}$ can be written as
\begin{align*}
\epsilon\dot{\mathcal{W}} & = (e_\Omega+ c_3 J^{-1}\bar e_R)\cdot (-k_R \bar e_R -k_\Omega e_\Omega)\\
&\quad + k_R \bar e_R\cdot e_\Omega + c_3 \dot e_R \cdot e_\Omega \leq - \zeta^T U \zeta,
\end{align*}
where the matrix $U\in\Re^{2\times 2}$ is 
\begin{align*}
U = \begin{bmatrix} \frac{c_3 k_R}{\lambda_M(J)} & -\frac{c_3 k_\Omega}{2\lambda_m(J)}\\
-\frac{c_3 k_\Omega}{2\lambda_m(J)} & k_\Omega-c_3
  \end{bmatrix}.
\end{align*}
The condition on $c_3$ guarantees that all of matrices $L_1,L_2,U$ are positive-definite. Therefore, the zero equilibrium of the tracking errors $(\bar e_R,e_\Omega)$ is exponentially stable, and the convergence rate is proportional to $\frac{1}{\epsilon}$. 

Next, we consider the \textit{reduced system}, which corresponds the full dynamic model when $R\equiv R_{c}$. From \refeqn{fi} and \refeqn{b3c}, the total thrust of quadrotor when $R=R_{c}$ is given by
\begin{align*}
-A \cdot R e_3 = (A\cdot R_{c} e_3) R_{c} e_3 = 
(A\cdot -\frac{A}{\|A\|}) -\frac{A}{\|A\|} = A.
\end{align*}
Therefore, the reduced system is given by the controlled dynamics for the simplified dynamic model discussed at Section \ref{sec:LCS}, which is exponentially stable. 

Then, according to Tikhonov's theorem~\cite[Thm 9.3]{Kha96}, there exists $\epsilon^* >0$ such that for all $\epsilon<\epsilon^*$, the origin of the full dynamics model is exponentially stable. 

\bibliography{BibMaster}

\begin{thebibliography}{10}
\providecommand{\url}[1]{#1}
\csname url@rmstyle\endcsname
\providecommand{\newblock}{\relax}
\providecommand{\bibinfo}[2]{#2}
\providecommand\BIBentrySTDinterwordspacing{\spaceskip=0pt\relax}
\providecommand\BIBentryALTinterwordstretchfactor{4}
\providecommand\BIBentryALTinterwordspacing{\spaceskip=\fontdimen2\font plus
\BIBentryALTinterwordstretchfactor\fontdimen3\font minus
  \fontdimen4\font\relax}
\providecommand\BIBforeignlanguage[2]{{%
\expandafter\ifx\csname l@#1\endcsname\relax
\typeout{** WARNING: IEEEtran.bst: No hyphenation pattern has been}%
\typeout{** loaded for the language `#1'. Using the pattern for}%
\typeout{** the default language instead.}%
\else
\language=\csname l@#1\endcsname
\fi
#2}}

\bibitem{LeeKimIJCAS09}
D.~Lee, H.~Kim, and S.~Sastry, ``Feedback linearization vs. adaptive sliding
  mode control for a quadrotor helicopter,'' \emph{International Journal of
  Control, Automation, and Systems}, vol.~7, no.~3, pp. 419--428, June 2009.

\bibitem{TayMcGITCSTI06}
A.~Tayebi and S.~McGilvray, ``Attitude stabilization of a {V}{T}{O}{L}
  quadrotor aircraft,'' \emph{IEEE Transactions on Control System Technology},
  vol.~14, no.~3, pp. 562--571, 2006.

\bibitem{PouMahACRA06}
P.~Pounds, R.~Mahony, and P.~Corke, ``Modeling and control of a large quadrotor
  robot,'' \emph{Control Engineering Practice}, vol.~18, pp. 691--699, 2010.

\bibitem{LeeLeoPICDC10}
T.~Lee, M.~Leok, and N.~McClamroch, ``Geometric tracking control of a quadrotor
  {UAV} on {SE(3)},'' in \emph{Proceedings of the IEEE Conference on Decision
  and Control}, 2010, pp. 5420--5425.

\bibitem{GilHuaPICRA10}
J.~Gillula, H.~Huang, M.~Vitus, and C.~Tomlin, ``Design of guaranteed safe
  maneuvers using reachable sets: Autonomous quadrotor aerobatics in theory and
  practice,'' in \emph{Proceedings of the International Conference on Robotics
  and Automation}, 2010, pp. 1649--1654.

\bibitem{MelMicIJRR12}
D.~Mellinger, N.~Michael, and V.~Kumar, ``Trajectory generation and control for
  precise aggressive maneuvers with quadrotors,'' \emph{International Journal
  Of Robotics Research}, vol.~31, no.~5, pp. 664--674, 2012.

\bibitem{CicKanJAHS95}
L.~Cicolani, G.~Kanning, and R.~Synnestvedt, ``Simulation of the dynamics of
  helicopter slung load systems,'' \emph{Journal of the American Helicopter
  Society}, vol.~40, no.~4, pp. 44--61, 1995.

\bibitem{BerPICRA09}
M.~Bernard, ``Generic slung load transportation system using small size
  helicopters,'' in \emph{Proceedings of the International Conference on
  Robotics and Automation}, 2009, pp. 3258--3264.

\bibitem{PalCruIRAM12}
I.~Palunko, P.~Cruz, and R.~Fierro, ``Agile load transportation,'' \emph{IEEE
  Robotics and Automation Magazine}, vol.~19, no.~3, pp. 69--79, 2012.

\bibitem{MicFinAR11}
N.~Michael, J.~Fink, and V.~Kumar, ``Cooperative manipulation and
  transportation with aerial robots,'' \emph{Autonomous Robots}, vol.~30, pp.
  73--86, 2011.

\bibitem{MazKonJIRS10}
I.~Maza, K.~Kondak, M.~Bernard, and A.~Ollero, ``Multi-{UAV} cooperation and
  control for load transportation and deployment,'' \emph{Journal of
  Intelligent and Robotic Systems}, vol.~57, pp. 417--449, 2010.

\bibitem{MelShoDARSSTAR13}
D.~Mellinger, M.~Shomin, N.~Michael, and V.~Kumar, ``Cooperative grasping and
  transport using multiple quadrotors,'' \emph{Distributed Autonomous Robotic
  Systems, Springer Tracts in Advanced Robotics}, vol.~83, pp. 545--558, 2013.

\bibitem{ZamStaJDSMC08}
D.~Zameroski, G.~Starr, J.~Wood, and R.~Lumia, ``Rapid swing-free transport of
  nonlinear payloads using dynamic programming,'' \emph{Journal of Dynamic
  Systems, Measurement, and Control}, vol. 130, no.~4, p. 041001, June 2008.

\bibitem{SchMurIICRA12}
J.~Schultz and T.~Murphey, ``Trajectory generation for underactuated control of
  a suspended mass,'' in \emph{IEEE International Conference on Robotics and
  Automation}, May 2012, pp. 123--129.

\bibitem{PalFieIICRA12}
I.~Palunko, R.~Fierro, and P.~Cruz, ``Trajectory generation for swing-free
  maneuvers of a quadrotor with suspended payload: A dynamic programming
  approach,'' in \emph{IEEE International Conference on Robotics and
  Automation}, RiverCentre, Saint Paul, Minnesota, USA, May 14-18 2012.

\bibitem{SreMicPICRA13}
K.~Sreenath, N.~Michael, and V.~Kumar, ``Trajectory generation and control of a
  quadrotor with a cable-suspended load - a differentially-flat hybrid
  system,'' in \emph{Proceedings of the International Conference on Robotics
  and Automation}, 2013, accepted.

\bibitem{SreLeePICDC13}
K.~Sreenath, T.~Lee, and V.~Kumar, ``Geometric control and differential
  flatness of a quadrotor {UAV} with a cable-suspended load,'' in
  \emph{Proceedings of the IEEE Conference on Decision and Control}, 2013,
  accepted.

\bibitem{LeeSrePICDC13}
T.~Lee, K.~Sreenath, and V.~Kumar, ``Geometric control of cooperating multiple
  quadrotor {UAV}s with a suspended load,'' in \emph{Proceedings of the IEEE
  Conference on Decision and Control}, 2013, accepted.

\bibitem{LeeLeoPICDC12}
T.~Lee, M.~Leok, and N.~McClamroch, ``Dynamics and control of a chain pendulum
  on a cart,'' in \emph{Proceedings of the IEEE Conference on Decision and
  Control}, 2012, pp. 2502--2508.

\bibitem{LeeLeoAJC12}
------, ``Nonlinear robust tracking control of a quadrotor {UAV} on {SE(3)},''
  \emph{Asian Journal of Control}, 2012, accepted.

\bibitem{GooLeePECC13}
F.~Goodarzi, D.~Lee, and T.~Lee, ``Geometric nonlinear {PID} control of a
  quadrotor {UAV} on {$\SE$},'' in \emph{Proceedings of the European Control
  Conference}, Zurich, July 2013.

\bibitem{Lee08}
T.~Lee, ``Computational geometric mechanics and control of rigid bodies,''
  Ph.D. dissertation, University of Michigan, 2008.

\bibitem{LeeLeoIJNME08}
T.~Lee, M.~Leok, and N.~H. McClamroch, ``Lagrangian mechanics and variational
  integrators on two-spheres,'' \emph{International Journal for Numerical
  Methods in Engineering}, vol.~79, no.~9, pp. 1147--1174, 2009.

\bibitem{Kha96}
H.~Khalil, \emph{Nonlinear Systems}, 2nd Edition, Ed.\hskip 1em plus 0.5em
  minus 0.4em\relax Prentice Hall, 1996.

\end{thebibliography}
\bibliographystyle{IEEEtran}

\end{document}